\documentclass[12pt, reqno]{article}
\usepackage{amssymb, amsthm, amsmath, amsfonts}
\usepackage{array, epsfig, subfig}
\usepackage{bbm}

\evensidemargin \oddsidemargin
\begin{document}

\newcommand{\be}{\begin{equation}}
\newcommand{\ee}{\end{equation}}
\newcommand{\bea}{\begin{eqnarray}}
\newcommand{\eea}{\end{eqnarray}}
\newcommand{\beaa}{\begin{eqnarray*}}
\newcommand{\eeaa}{\end{eqnarray*}}

\renewcommand{\proofname}{\bf Proof}
\newtheorem*{rem*}{Remark}
\newtheorem*{cor*}{Corollary}
\newtheorem{prop}{Proposition}
\newtheorem{lem}{Lemma}
\newtheorem{theo}{Theorem}
\newfont{\zapf}{pzcmi}

\def\R{\mathbb{R}}
\def\N{\mathbb{N}}
\def\E{\mathbb{E}}
\def\P{\mathbb{P}}
\def\V{\mathbb{D}}
\def\I{\mathbbm{1}}

\def\al{\alpha}
\def\bu{$\bullet$}
\newcommand{\D}{\hbox{ \zapf D}}

\title{Probabilities of competing binomial random variables}

\author{Wenbo V. Li\thanks{Department of Mathematical Sciences, University of Delaware,
\texttt{wli@math.udel.edu}. Supported in part by NSF grant DMS--0805929, NSFC-6398100, CAS-2008DP173182.}
\and
Vladislav V. Vysotsky\thanks{This work started at the University of Delaware. The current affiliations are School of Mathematical and Statistical Sciences, Arizona State University; St.Petersburg Division of Steklov Mathematical Institute; and Chebyshev Laboratory at St.Petersburg State University, \texttt{vysotsky@asu.edu}. Supported in part by the grant NSh. 4472-2010-1.}}

\maketitle

\begin{abstract}
Suppose you and your friend both do $n$ tosses of an unfair coin with probability of heads equal to $\alpha$. What is the behavior of the
probability that you obtain at least $d$ more heads than your friend if you make $r$ additional tosses?
We obtain asymptotic and monotonicity/convexity properties for this competing probability as a function of $n$,
and demonstrate surprising phase transition phenomenons as parameters $ d, r$ and $\alpha$ vary.
Our main tools are integral representations based on Fourier analysis.



\end{abstract}

\section{Introduction}

Suppose you and your friend both do $n$ tosses of an unfair coin with probability of heads equal to $\alpha$. What is the behavior of the competing
probability that you obtain at least $d$ more heads than your friend if you make $r$ additional tosses?


For a fair coin with $\alpha=1/2$ and $r=d=1$, this is Example 3.33 on page 118 of the textbook \cite{G}:
``Adam tosses a fair coin $n+1$ times, Andrew tosses the same coin $n$ times. What is the probability that Adam gets more heads than Andrew?''

Two solutions are offered in the textbook. One uses the symmetry and finds the answer $1/2$ easily.
The other is a direct computation, following the identity
\begin{equation} \label{comb-2}
\sum_{i=0}^n \sum_{j=i+1}^{n+1} {n+1 \choose j} {n \choose i} =2^{2n},
\end{equation}
and the author noted that ``a combinatorial solution to this problem is neither elegant nor easy to handle''.

The same problem also appeared as Problem 21 of the self-test problems and exercises on page 115 of the textbook \cite{R}, and asks the following: ``If A flips $n+1$ and B flips $n$ fair coins, show that the probability that A gets more heads than B is $1/2$.'' This problem is at the end of the chapter on conditional probability and independence,
and with the hint that one should ``condition on which player has more heads after each has flipped $n$ coins.''

What happens if the coin is unfair is not mentioned in both textbooks.
In this paper, we consider competing probability of two independent binomials and the associated phase transition behaviors as parameters vary.

To pose the problem formally, let $S_k$ and $S_n'$ be independent binomial random variables, that is, $S_k = X_1 + \dots + X_k$ and $S_n' = X_1' + \dots + X_n'$, where
$\{ X_i, X_i' \}_{i \ge 1}$ are independent identically distributed random variables that equal $1$ with probability $\alpha$ and $0$ with probability $1-\alpha$. The textbooks consider the probabilities $$p_n:=\P \bigl\{ S_{n+1} \ge S_n' +1 \bigr\}$$ while we are interested in the more general
$$
p_n^{r , d }:=\P \bigl\{ S_{n+r} \ge S_n' + d\bigr\},
$$
where of course $p_n^{1 ,1}=p_n$; we always assume that $\alpha \in (0,1)$ and $r, d \ge 1$.

Clearly we can write a combinatorial expression
\begin{eqnarray}
p_n^{r , d }&=&\sum_{i=0}^n \sum_{j=i+d}^{n+r}\P(S_{n+r}=j)\cdot \P(S_n'=i) \nonumber \\
&=&\sum_{i=0}^n \sum_{j=i+d}^{n+r} {n+r \choose j} {n \choose i}\alpha^{i+j}(1-\alpha)^{2n+r-i-j} \label{comb}
\end{eqnarray}
by using independence and binomial probabilities. Unlike (\ref{comb-2}), the expression (\ref{comb}) for $p_n^{r , d }$ is of a little use in analyzing finer behaviors as parameters vary. On the other hand, by the Cental Limit Theorem, for any $r, d$ and $\alpha$ as above,
\begin{equation} \label{clt 1/2}
\lim_{n \to \infty} p_n^{r , d } = \lim_{n \to \infty} \P \Bigl\{ \frac{S_{n+r}}{\sqrt{n}} \ge \frac{S_n' + d}{\sqrt{n}} \Bigr\}=\P \bigl\{ N > N' \bigr\} =\frac12,
\end{equation}
where $N$ and $N'$ are independent normal random variables with mean $0$ and variance $\alpha (1-\alpha)$.
This standard technique allows one to find the limit but tells nothing about the mode and exact rate of convergence. Our first result provides 
a useful integral representation for  $p_n^{r , d }$ which implies precise asymptotic.

\begin{theo}\label{MAIN TH-1}
For any positive integers $r$ and $d$,
\begin{equation} \label{f12}
p_n^{r , d } = \frac12+\frac{1}{\pi}\int_0^1 Q_\alpha^n(x) (1-x^2)^{-1/2} P_{\alpha}^{r, d} (x) dx
\end{equation}
where
\begin{equation}\label{Qa}
Q_\alpha(x):=1 - 4 \alpha (1-\alpha) (1 - x^2)
\end{equation}
and
\begin{equation}\label{Pa}
P_\alpha^{r, d} (x) :=  \sum_{j=0}^{r+d} \frac{d^{2j}}{d
t^{2j}} \Bigl( t^{j-d} (1 - \alpha + \alpha t)^r \Bigr) \Bigl. \Bigr |_{t=-1} \cdot
\frac{(2x)^{2j}}{(2j)!}.
\end{equation}
As a consequence,
\be\label{sqrta}
\lim_{n \to \infty}\sqrt{n} \bigl(p_n^{r,d}-{1 \over 2}\bigr)= \frac{1}{4\sqrt{\pi}} \frac{2 \alpha r-2d+1}{\sqrt{\alpha(1-\alpha)}}.
\ee
\end{theo}

The actual degree of $P_\alpha^{r,d}(x)$ is $2 \max (r-d, d-1)$ as the higher coefficients in \eqref{Pa} vanish. In Section~\ref{ANALYSIS} we give another convenient formula for $P_{\alpha}^{r, d}(x)$ in terms of Chebyshev polynomials of the second kind, and then easily show that
$$P_{\alpha}^{r, d}(1)=2 \alpha r-2d+1.
$$
Note that for any $\alpha \in (0,1)$, the function $Q_\alpha(x)$ is increasing on $[0,1]$ with $Q_\alpha(0) \ge 0$ and $Q_\alpha(1)=1$. Therefore as $n$ increases, the main contribution to the integral in \eqref{f12} comes from a decreasing small neighborhood of $1$, and a standard analysis of \eqref{f12} implies \eqref{sqrta}. Another advantage of the integral representation (\ref{f12}) stems from the isolation of the variable $n$ from parameters $r$ and $d$.

Since the integral representation for $p_n^{r,d}$ in (\ref{f12}) can be integrated out via trigonometric substitution $x=\cos (t/2)$, it seems possible to check the equivalence of (\ref{comb}) and (\ref{f12}) by pure algebraic manipulations. However, such an approach provides no probabilistic insights into the integral representation and no clues on how we discovered it. We will present a proof via a combination of one-step-back analysis and Fourier analytic methods, along the line we initially derived the representation.

Next we consider the mode of convergence in terms of monotonicity and convexity properties in the simplest setting $r=d=1$. Let us agree that ``increasing/decreasing'' stand for ``strictly increasing/strictly decreasing'' throughout this paper.

\begin{theo} \label{MAIN TH-2}
The sequence $p_n=p_n^{1, 1}=\P \{ S_{n+1} \ge S'_n+1 \}= \P \{ S_{n+1} > S'_n \}, \, n = 1,2, \dots,$ is monotone and convex/concave. Precisely,
$p_n$ is increasing and concave when $\alpha < 1/2$,  decreasing and convex when $\alpha > 1/2$, and equal to $1/2$ for all $n$ when $\alpha = 1/2$.
\end{theo}

\begin{figure}[ht]
\centering
\includegraphics[width=0.32\textwidth]{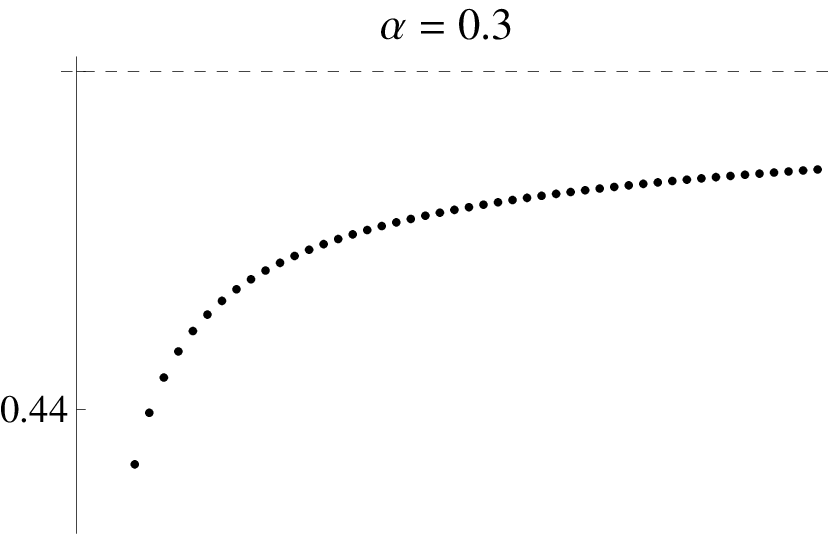}
\includegraphics[width=0.32\textwidth]{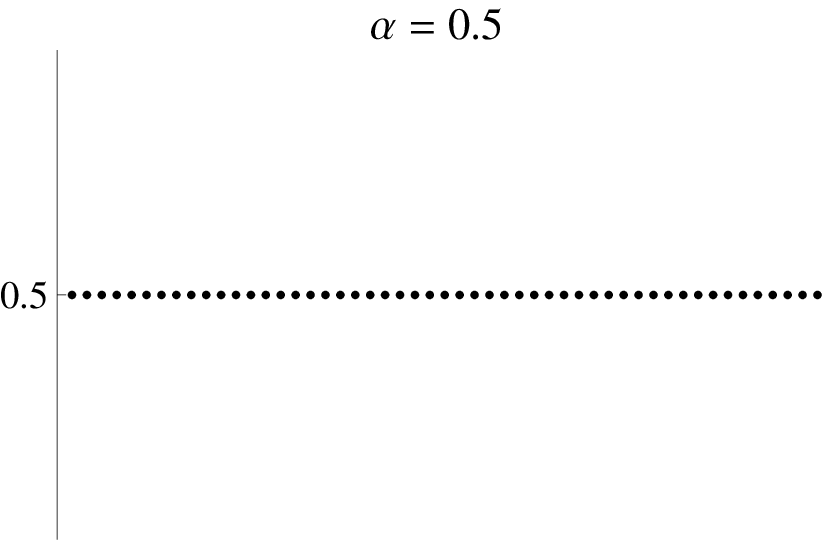}
\includegraphics[width=0.32\textwidth]{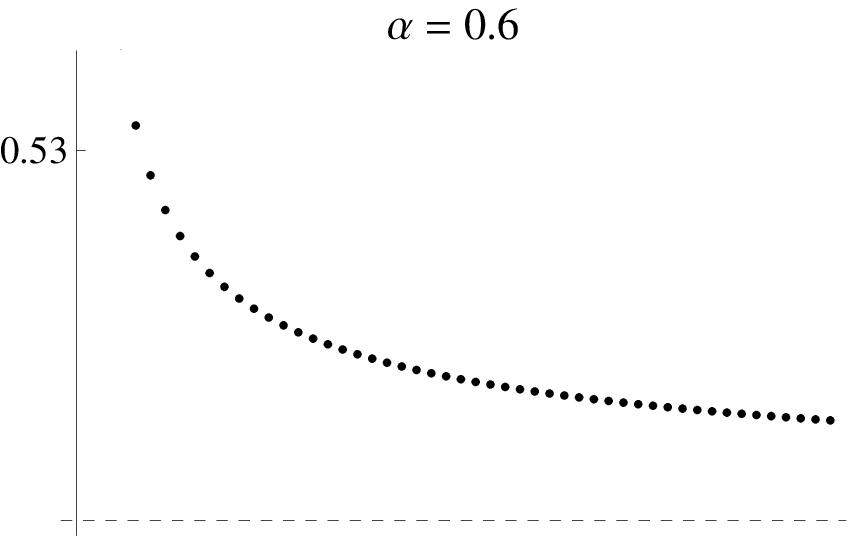}
\caption{The sequence $p_n$, $0 \le n \le 50$, for different $\alpha$. The dashed line corresponds to 0.5.} \label{Fig 1}
\end{figure}

It is remarkable that the sequence $p_n^{1, 1}$ is monotone starting from its first term. We first discovered this feature when studying some very specific probabilistic properties of the so-called double-sided exponential random variables, see the discussion at the end of the paper.

Observe that as $\alpha$ decreases from $1$ to $0$, the behavior of the sequence $p_n^{1, 1}$ changes instantly at the critical value $1/2$, see Fig.~\ref{Fig 1}. This phenomenon becomes much clear in the setting with $r \ge 2$ and $d=1$. In order to give the best statement of our results, from this point on we assume that $n \ge 0$ rather than $n \ge 1$, which was natural for the introduction of the problem. We say that a sequence $a_n, n \ge 0,$ is unimodal with the mode $N \ge 1$ if $a_0 < a_1 < \dots < a_{N-1} \le a_N > a_{N+1} > a_{N+2} > \dots$ .

\begin{theo} \label{MAIN 3}
For $r \ge 2$, the sequence $p_n^{r, 1} $ is increasing when $0 < \alpha \le \frac{1}{2r}$,
unimodal when $\frac{1}{2r} < \alpha \le \frac{1}{r+1}$, and decreasing when $\frac{1}{r+1} < \alpha < 1$. In addition, the mode $N_{\alpha, r} $ satisfies
\begin{equation} \label{mode}
N_{\alpha, r}  \sim \frac{r-1}{4r} \Bigl(\alpha-{1 \over 2r} \Bigr)^{-1} \quad \hbox{as} \quad \alpha \downarrow {1 \over 2r}
\end{equation}
while for the maximum
\begin{equation} \label{maximum}
p_{N_\alpha,r} - \frac12 \sim \frac43 \sqrt{\frac{ r^5 }{ \pi (2r-1)(r-1)}} \Bigl(\alpha-{1 \over 2r} \Bigr)^{3/2} \quad \hbox{as} \quad \alpha \downarrow {1 \over 2r}.
\end{equation}
\end{theo}

\begin{figure}[ht]
\centering

\subfloat[]{\label{fig(a)}\includegraphics[width=0.32\textwidth]{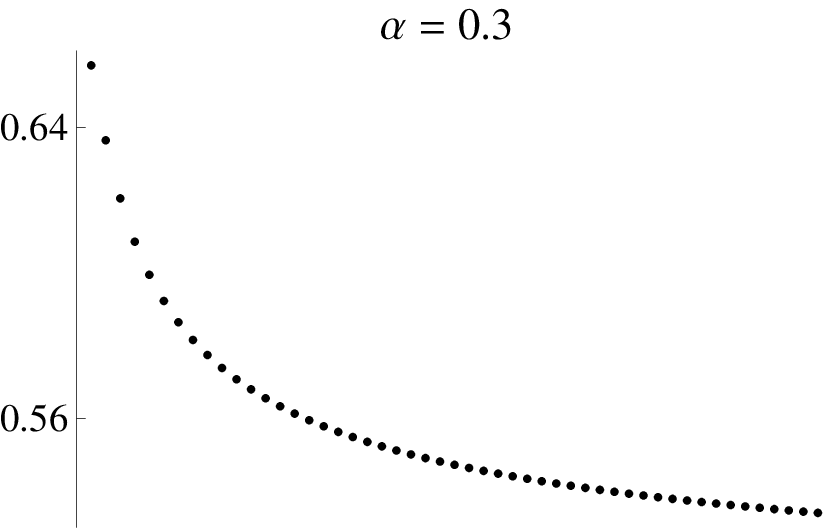}}
\subfloat[]{\label{fig(b)}\includegraphics[width=0.32\textwidth]{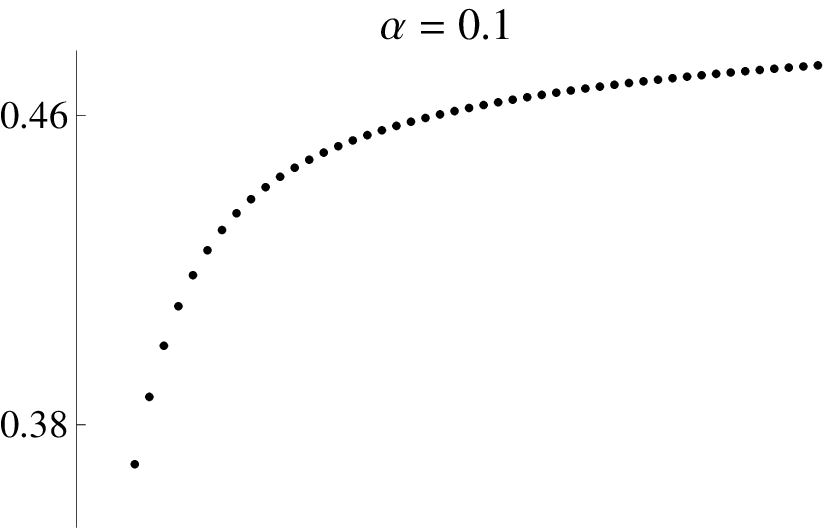}}

\subfloat[]{\label{fig(c)}
\includegraphics[width=0.32\textwidth]{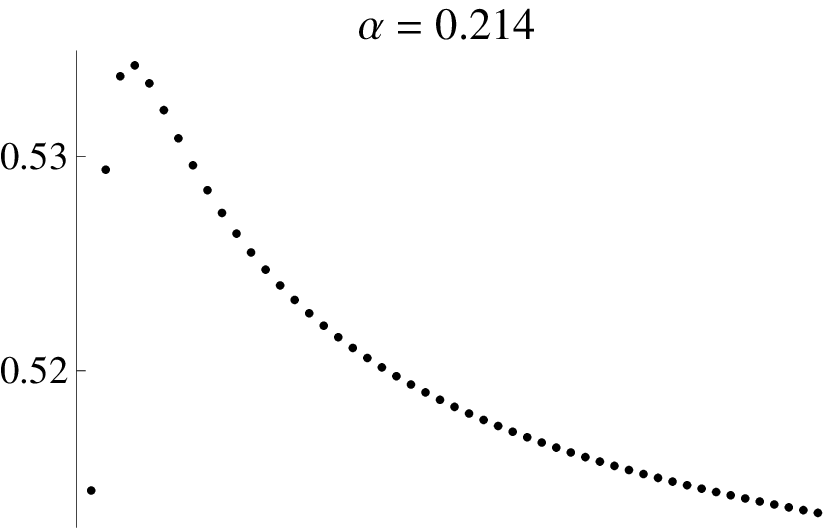}
\includegraphics[width=0.32\textwidth]{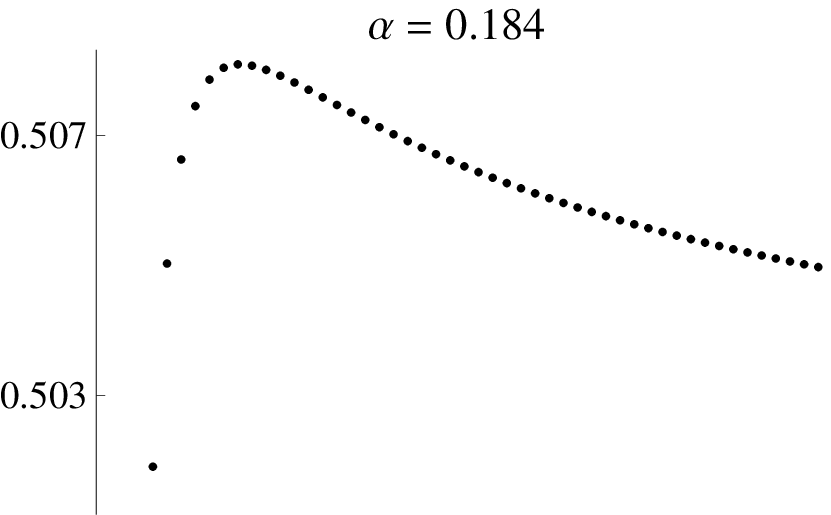}
\includegraphics[width=0.32\textwidth]{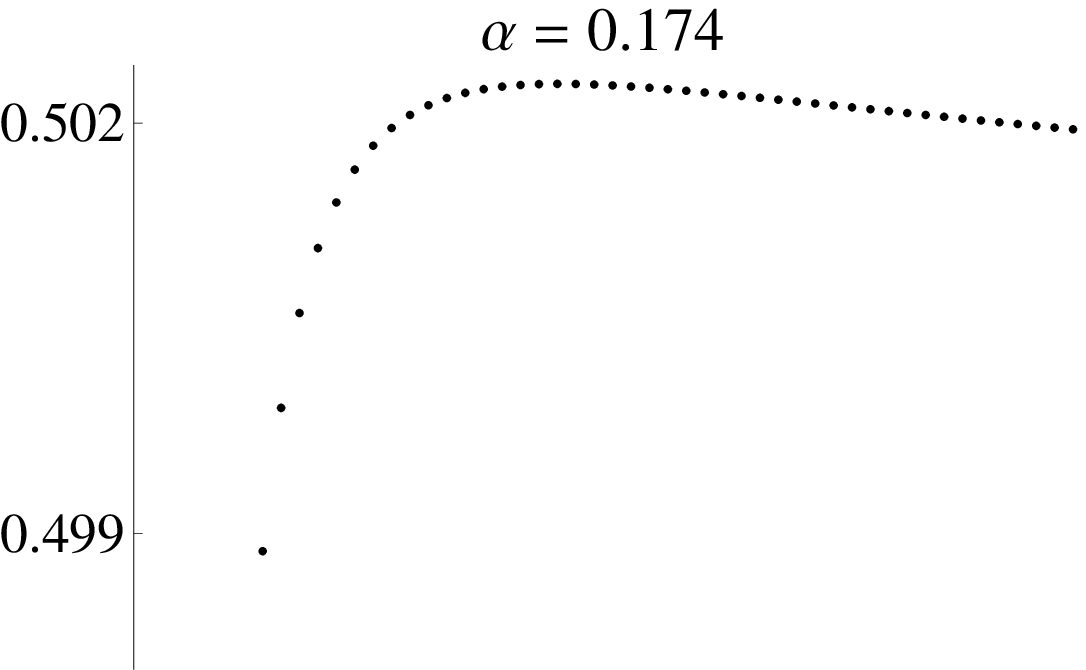}}
\caption{The sequence $p_n^{3, 1}$, $0 \le n \le 50$: monotone modes (a) and (b) and the phase transition (c) between them with decrease of $\alpha$. In the transitional mode (c) the point of maximum drifts to the right.} \label{Fig 2}
\end{figure}

We observe the peculiar ``phase transition'' in the behavior of $p_n^{r, 1}$ with decrease of $\alpha$, see Fig.~\ref{Fig 2}. For large $\alpha$ the sequence is decreasing while as $\alpha$ gets smaller reaching the critical value $1/(r+1)$, $p_n^{r, 1}$ becomes unimodal and the point of maximum drifts to the right as $\alpha$ decays to $1/(2r)$, as shown in Fig.~\ref{fig(c)}. When $\alpha$ reaches the critical value $1/(2r)$, $p_n^{r, 1}$ becomes increasing, which corresponds to the limit case $N_{\alpha, r}  = \infty$. It is indeed remarkable that in the transitional mode the sequence $p_n^{r, 1}$ is unimodal. Relations \eqref{mode} and \eqref{maximum} of Theorem~\ref{MAIN 3} describe the ``speed'' of phase transition near the critical value $\alpha=1/(2r)$.

For the general setting as parameters $\alpha, d, r$ vary, we have the following slightly less precise results
which still capture the phase transition phenomenon between the monotone modes of convergence.

\begin{theo} \label{MAIN 2}
Depending on the relations between $r$ and $d$, the sequence $p_n^{r, d}$ is

\begin{enumerate}
\item increasing (for any $0 < \alpha <1$) if $1 \le r \le d-1$;

\item
\begin{enumerate}
\item increasing (for all $n$) when $0 < \alpha < d/(r+1)$,
\item increasing for large $n$ when $d/(r+1) \le \alpha < (2d-1)/(2 r)$,
\item decreasing for large $n$ when $(2d-1)/(2 r) \le \alpha < 1$,
\end{enumerate}
if $d \le r \le 2d-2$;

\item
\begin{enumerate}
\item increasing when $0 < \alpha < 1/2$,
\item identically equals $1/2$ when $\alpha = 1/2$,
\item decreasing when $1/2  < \alpha < 1$,
\end{enumerate}
if $r = 2d-1$.

\item
\begin{enumerate}
\item increasing for large $n$ when $0 <  \alpha \le (2d-1)/(2 r)$,
\item decreasing for large $n$ when $(2d-1)/(2 r)  < \alpha \le d/(r+1)$,
\item decreasing (for all $n$) when $d/(r+1) < \alpha <1$,
\end{enumerate}
if $r \ge 2d $.
\end{enumerate}
\end{theo}

\begin{figure}[ht]
\centering
\includegraphics[width=0.32\textwidth]{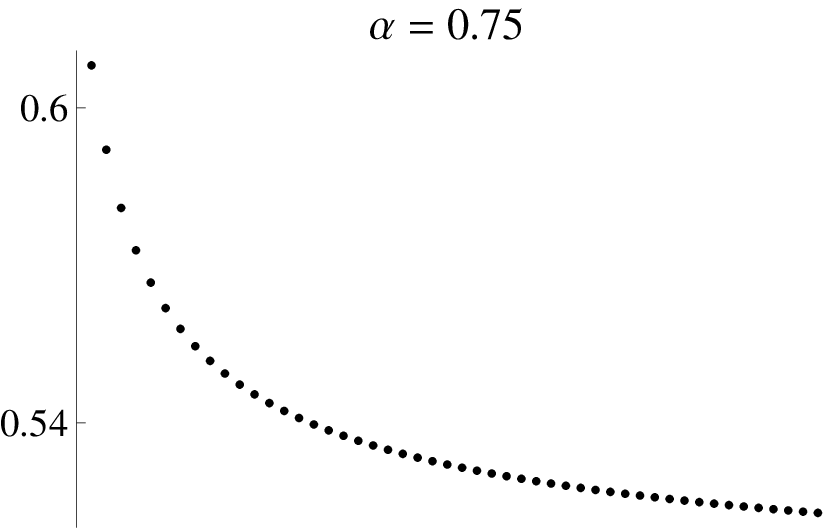}
\includegraphics[width=0.32\textwidth]{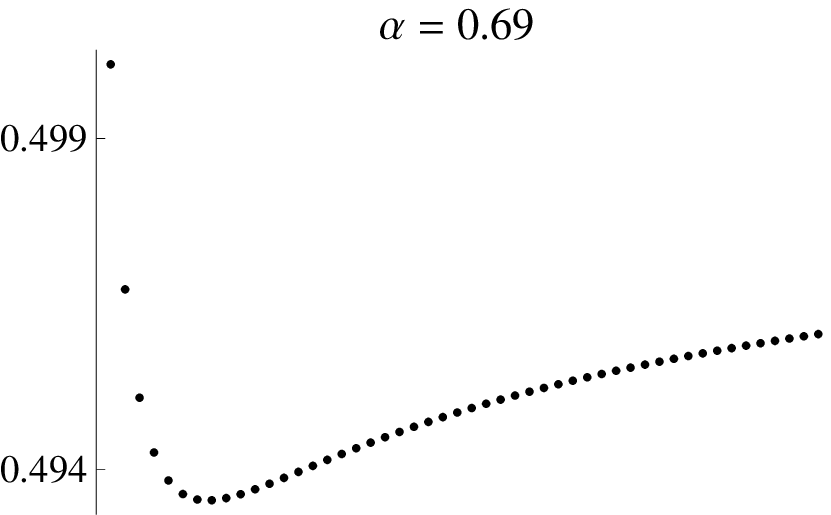}
\includegraphics[width=0.32\textwidth]{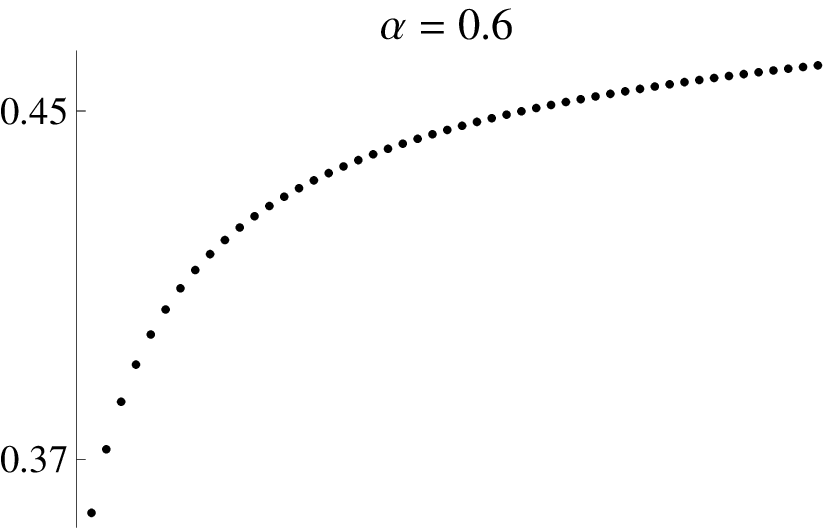}
\caption{Phase transition in the behavior of $p_n^{5, 4}$, $0 \le n \le 50$, with decrease of $\alpha$.} \label{Fig 3}
\end{figure}

Part 1 is in some sense degenerate while Part 2 uncovers another type of behavior of $p_n^{r, d}$, see Fig.~\ref{Fig 3}. However, the reader may concentrate on the case $r \ge 2d-1$, which is already familiar from Theorems~\ref{MAIN TH-2}~and~\ref{MAIN 3}, as all the results Part 2 for $d \le r \le 2d-2$ follow from Part 4 by a certain duality relation given in Section~\ref{PROOFS}.

\section{One-step-back analysis and Fourier method} \label{ANALYSIS}

It is natural to look into the increments of $p_n^{r,d}$ in $n$. We start with some elementary but very useful consideration, one-step-back analysis.

\begin{lem} \label{DIFF p_n^r,d 1} For any $r, d \ge 1$,
$$p_{n+1}^{r , d } - p_n^{r , d } = \alpha (1-\alpha) \sum_{i=0}^r \P \bigl \{ S_r = i \bigr \} \bigl( q_n^{(i-d+1)} - q_n^{(i-d)} \bigr)$$
where the sequence, defined by
$$q_n^{(i)} := \P \bigl \{ S_n' - S_n =i \bigr\},$$
is decreasing in $i$ for $i \ge 0$ and satisfies $q_n^{(-i)} = q_n^{(i)}$.
\end{lem}
\begin{proof}
We write
$$p_n^{r , d }=\P \bigl\{ S_{n+r} \ge S_n' + d \bigr\} = \P \bigl\{ S_r'' \ge S_n' - S_n + d \bigr\} = \sum_{i=-n}^n q_n^{(i)} \P \bigl \{ S_r \ge i + d \bigr \}$$ with $S_r'':=S_{n+r}-S_r$ and similarly,
$$p_{n+1}^{r , d }=\P \bigl\{ S_r'' \ge S_n' + X_{n+1}' - S_n - X_{n +1} + d  \bigr\} = \sum_{i=-n}^n q_n^{(i)} \P \bigl \{ S_r \ge X_1' - X_1'' + i + d \bigr \}.$$ Hence
\begin{eqnarray*}
p_{n+1}^{r , d } - p_n^{r , d } &=& \alpha (1-\alpha) \sum_{i=-n}^n q_n^{(i)} \Bigl (\P \bigl \{ S_r \ge i + d + 1 \bigr \} + \P \bigl \{ S_r \ge i + d - 1 \bigr \} - 2 \P \bigl \{ S_r \ge i + d  \bigr \} \Bigr) \\
&=& \alpha (1-\alpha) \sum_{i=-n}^n q_n^{(i)} \Bigl ( \P \bigl \{ S_r = i + d - 1 \bigr \} - \P \bigl \{ S_r = i + d \bigr \} \Bigr),
\end{eqnarray*}
which yields the first statement of the lemma.

The relation $q_n^{(-i)} = q_n^{(i)}$ is obvious. Let us use induction to show that $q_n^{(i)} > q_n^{(i+1)}$ for any $n$ and $0 \le i \le n$.
The initial case $n=1$ is trivial. We use the law of total probability to get from $n$ to $n+1$: $$q_{n+1}^{(i)} = \alpha (1-\alpha) \Bigl( q_n^{(i-1)} +  q_n^{(i+1)}\Bigr) + \bigl( \alpha^2 + (1-\alpha)^2 \bigr) q_n^{(i)}.$$
Then $$q_{n+1}^{(i)} - q_{n+1}^{(i+1)} = \alpha (1-\alpha) \Bigl( q_n^{(i-1)} -
q_n^{(i+2)}\Bigr) + \Bigl( \alpha^2 + (1-\alpha)^2 - \alpha(1-\alpha) \Bigr) \Bigl(
q_n^{(i)} - q_n^{(i+1)}\Bigr),$$ which is positive by the induction hypothesis.

\end{proof}

In order to gain more quantitative information, we employ the powerful Fourier method. Recall that for any integer-valued random variable $Z$, it holds that
\begin{equation}
\label{Fourier}
\P \bigl\{ Z=0\bigr\} = \frac{1}{2 \pi} \int_{-\pi}^{\pi} \E e^{itZ} dt.
\end{equation}
Indeed, $$\int_{-\pi}^{\pi} \E e^{itZ} dt = \int_{-\pi}^{\pi} \sum_{k=-\infty}^{\infty} e^{itk} \P \bigl\{ Z=k\bigr\} dt = \sum_{k=-\infty}^{\infty} \int_{-\pi}^{\pi} e^{itk} dt \P \bigl\{ Z=k\bigr\} = 2 \pi \P \bigl\{ Z=0\bigr\}.$$

\begin{lem} \label{DIFF p_n^r,d 2} For any $r, d \ge 1$,
$$p_{n+1}^{r , d } - p_n^{r , d } = -\frac{4 \alpha (1-\alpha)}{\pi} \int_0^1 Q_\alpha^n(x) \sqrt{1-x^2} \tilde{P}_{\alpha}^{r, d} (x) dx,$$
where
\begin{equation} \label{P tilda}
\tilde{P}_{\alpha}^{r, d}(x):= \sum_{i=d}^r \P \bigl \{ S_r = i \bigr \} U_{2(i-d)}(x) - \sum_{i=0}^{d-1} \P \bigl \{ S_r = i \bigr \} U_{2(d-i-1)}(x)
\end{equation}
is defined in terms of Chebyshev polynomials $U_k(x)$ of the second kind.

\end{lem}

Recall that Chebyshev polynomials of the second kind, defined by $U_k(x)=\frac{\sin(k+1) \theta}{\sin \theta}$ with $x=\cos \theta$, are orthogonal on $[-1,1]$ with
weight $\sqrt{1-x^2}$. It is remarkable that $\tilde{P}_{\alpha}^{r, d} (x)$ is explicitly expressed in terms of $U_k(x)$, whose properties are well known.
In the next section we will show that $\tilde{P}_{\alpha}^{r, d} (x) = P_{\alpha}^{r, d} (x) $, and thus \eqref{P tilda} actually serves as another useful representation of the polynomials defined in Theorem~\ref{MAIN TH-1}.


\begin{proof}

With the reminder $q_n^{(k)} = \P \bigl \{ S_n' - S_n =k \bigr\}$, we use \eqref{Fourier} to write
\begin{eqnarray*}
q_n^{(k)}  &=& \frac{1}{2 \pi} \int_{-\pi}^{\pi} \E e^{it(S_n - S_n' - k)} dt = \frac{1}{2 \pi} \int_{-\pi}^{\pi} \varphi_\alpha^n(t) e^{-ikt}dt,
\end{eqnarray*}
where
\begin{eqnarray} \label{phi=}
\varphi_\alpha(t)&:=& \bigl(1-\alpha + \alpha e^{it}\bigr)\bigl(1-\alpha + \alpha e^{-it}\bigr) \notag\\
&=& 1 - 4 \alpha (1-\alpha) \sin^2 \frac{t}{2}
\end{eqnarray}
is the characteristic function of $S_1-S_1'$. By the symmetry of $\varphi_\alpha(t)$,
\begin{equation} \label{q_n=}
q_n^{(k)} = \frac{1}{\pi} \int_0^{\pi} \varphi_\alpha^n(t) \cos kt \, dt.
\end{equation}
and combining \eqref{q_n=} with Lemma~\ref{DIFF p_n^r,d 1}, we get
\begin{equation} \label{diff p_n^r,d cos}
p_{n+1}^{r , d } - p_n^{r , d } = \frac{\alpha (1-\alpha)}{\pi} \int_0^\pi
\varphi_\alpha^n(t) \sum_{i=0}^r \P \bigl \{ S_r = i \bigr \} \bigl( \cos (i-d+1)t - \cos
(i-d)t \bigr) d t.
\end{equation}
As for $k \ge 0$ it holds that
$$\cos (k+1) t - \cos kt = -2 \sin^2 \frac{t}{2} \cdot \frac{\sin (k+1/2)t}{\sin (t/2)} = - 2
\sin^2 \frac{t}{2} \cdot U_{2k} \Bigl (\cos \frac{t}{2} \Bigr),$$ \eqref{diff p_n^r,d cos} transforms to
$$p_{n+1}^{r , d } - p_n^{r , d } = -\frac{2 \alpha (1-\alpha)}{\pi} \int_0^{\pi}
\varphi_\alpha^n(t) \sin^2 \frac{t}{2} \cdot \tilde{P}_{\alpha}^{r, d} \Bigl(\cos \frac{t}{2}\Bigr) dt,$$
and we conclude the proof with the change $x= \cos (t/2)$.
\end{proof}

\section{Proofs of the theorems.} \label{PROOFS}

\begin{proof}[\bf Proof of Theorem~\ref{MAIN TH-1}]
First check that
\begin{equation} \label{P=P_tilde}
\tilde{P}_{\alpha}^{r, d} (x) = P_{\alpha}^{r, d} (x).
\end{equation}
Using the representation
$$U_{2k}(x)= \sum_{j=0}^k (-1)^j {2k-j \choose j} (2x)^{2k-2j} = \sum_{j=0}^k (-1)^{k-j} {k+j \choose 2j} (2x)^{2j}$$
from Section 6.10.7 of \cite{Z}, we obtain
\begin{eqnarray*}
\tilde{P}_\alpha^{r, d}(x) &=& \sum_{i=d}^{r} \sum_{j=0}^{i-d} \P \bigl \{ S_r = i \bigr \} (-1)^{i-j-d} {i+j-d \choose 2j} (2x)^{2j}\\
&& -\sum_{i=0}^{d-1} \sum_{j=0}^{d-i-1} \P \bigl \{ S_r = i \bigr \} (-1)^{d-i-j-1} {d-i+j-1 \choose 2j} (2x)^{2j}.\\
\end{eqnarray*}

Let us agree that ${n \choose k} := \frac{1}{k!} \prod_{m=n-k+1}^n m$ and ${n \choose 0}
:=1$ to write binomial coefficients with negative $n$. This allows us to note
that ${d-i+j-1 \choose 2j} = {i+j-d \choose 2j}$ for any $i, j, d \ge 0$. The next
important observation is that in both double sums the summation in $j$ could be taken
from $0$ to $\infty$. Indeed, any $j \ge i - d +1$ gives no contribution to the first sum as
the product $\prod_{m=i-j-d+1}^{i+j-d} m$, which corresponds to the binomial coefficient, includes a zero factor because
$i+j-d > 0$ and $i-j-d+1 \le 0$. The same applies to the second double sum for $j \ge d-i$ as the product $\prod_{m=d-i-j}^{d-i+j-1} m$ is zero since $d-i+j-1 >0$ while $d-i-j \le 0$. Thus
\begin{eqnarray*}
\tilde{P}_\alpha^{r, d}(x) &=& \sum_{i=0}^r \sum_{j=0}^\infty  \P \bigl \{ S_r = i
\bigr \} (-1)^{i-j-d} {i+j-d \choose 2j} (2x)^{2j}\\
&=&  \sum_{j=0}^{\infty} \frac{d^{2j}}{d t^{2j}} \Bigl( \E t^{S_r + j-d} \Bigr)
\Bigl. \Bigr |_{t=-1} \cdot \frac{(2x)^{2j}}{(2j)!}.
\end{eqnarray*}
Then \eqref{P=P_tilde} follows since the generating function of $S_r$ is $\E t^{S_r} = (1 - \alpha + \alpha t)^r$.

It now remains to use Lemma~\ref{DIFF p_n^r,d 2} and \eqref{P tilda} to write the telescoping sum
$$p_{n+k}^{r , d } - p_n^{r , d } = \sum_{i=0}^{k-1} p_{n+i+1}^{r , d } - p_{n+i}^{r , d } = -\frac{4 \alpha (1-\alpha)}{\pi} \int_0^1 \sum_{i=0}^{k-1} Q_\alpha^{n+i}(x) \sqrt{1-x^2} P_{\alpha}^{r, d} (x) dx.$$ As $k \to \infty$, $p_{n+k}^{r , d } \to 1/2$ by \eqref{clt 1/2}, and simplifying the sum of the geometric series with \eqref{Qa}, we get \eqref{f12} by the dominated convergence theorem.

For the precise asymptotic relation (\ref{sqrta}), we return to the convenient trigonometric substitution $x=\cos (t/2)$ in \eqref{f12} and get
$$p_n^{r , d } - \frac12 = \frac{1}{2\pi}\int_0^\pi \varphi_\alpha^n(t)  P_{\alpha}^{r, d} \Bigl(\cos \frac{t}{2}\Bigr) dt,$$ with $\varphi_\alpha(t) = Q_\alpha( \cos (t/2))=1-4\alpha(1-\alpha)\sin^2 (t/2)$. Choose a $\delta>0$ such that $t/4 \le \sin (t/2) \le t/2$ on $[0, \delta]$ and observe that $\varphi_\alpha(t)$ is decreasing on $[0, \pi]$ and $\varphi_\alpha(0)=1$. We have
\bea
\int_0^{\pi} \varphi_\alpha^n(t)  P_{\alpha}^{r, d} \Bigl(\cos \frac{t}{2}\Bigr) dt
&=& \int_0^{\delta} \varphi_\alpha^n(t)  P_{\alpha}^{r, d} \Bigl(\cos \frac{t}{2}\Bigr) dt + O \bigl( \varphi_\alpha^n(\delta) \bigr) \notag\\
&=& \frac{1}{\sqrt{n}} \int_0^{\delta \sqrt{n}} \varphi_\alpha^n \Bigl(\frac{s}{\sqrt{n}}\Bigr)  P_{\alpha}^{r, d} \Bigl(\cos \frac{s}{2\sqrt{n}}\Bigr) ds + O \bigl( \varphi_\alpha^n(\delta) \bigr) \label{dif sim},
\eea
where
\begin{equation} \label{phi lim}
\lim_{n \to \infty} \varphi_\alpha^n \Bigl(\frac{s}{\sqrt{n}}\Bigr) = \lim_{n \to \infty} \left( 1-4\alpha(1-\alpha)\sin^2 \frac{s}{2 \sqrt{n}} \right)^n  = e^{-\alpha(1-\al)s^2}.
\end{equation}
Now
\bea
\lim_{n \to \infty} \int_0^\infty \I_{[0, \delta \sqrt{n}]}(s) \varphi_\alpha^n \Bigl(\frac{s}{\sqrt{n}}\Bigr)  P_{\alpha}^{r, d} \Bigl(\cos \frac{s}{2\sqrt{n}}\Bigr)  ds
&=& \int_0^\infty e^{-\alpha(1-\alpha)s^2} P_{\alpha}^{r, d} (1)ds \label{Lebesgues}\\
&=& \frac12 \sqrt{\frac{\pi}{\alpha(1-\alpha)}}P_{\alpha}^{r, d} (1) \notag
\eea
by the dominated convergence theorem, with $e^{-\alpha(1-\alpha)s^2/4}\cdot \sup_{0 \le x \le 1}|P_\alpha^{r,d}(x)|$ as an integrable majorant. Combining the arguments above together, we finish the proof of Theorem 1 once we find $P_{\alpha}^{r, d} (1)$.

It is well known  that $U_k(1)=k+1$, hence \eqref{P tilda} and \eqref{P=P_tilde} imply
\bea
P_\alpha^{r, d}(1)  &=& \sum_{i=d}^r \P \bigl \{ S_r = i \bigr \} (2i-2d +1) -\sum_{i=0}^{d-1} \P \bigl \{ S_r = i \bigr \} (2d-2i-1)\nonumber \\
&=& \E (2S_r-2d+1)=2\alpha r - 2d + 1. \label{pard(1)}
\eea

\end{proof}

\begin{proof}[\bf Proof of Theorem~\ref{MAIN TH-2}]
As $P_\alpha^{1,1}(x)=2\alpha - 1$, Lemma~\ref{DIFF p_n^r,d 2} yields $$p_{n+1} - p_n = \frac{4\alpha (1-\alpha)(1 - 2 \alpha)}{\pi} \int_0^1 Q_\alpha^n(x) \sqrt{1-x^2} dx.$$ It is readily seen that $p_n$ is monotone because the
integrand is nonnegative; moveover, the integrand is monotone in $n$ for each $x$, so
$p_{n+1} - p_n$ is monotone implying convexity of $p_n$.
\end{proof}

Note that monotonicity of $p_n$ could be obtained directly from Lemma~\ref{DIFF p_n^r,d 1}, whose proof requires only an elementary one-step-back analysis. The same is true for convexity of $p_n$ but some additional study of properties of $q_n^{(i)}$ should be done.

It is also worth mentioning that the asymptotic of $p_n - 1/2$ could be found via purely probabilistic argument with no use of Fourier method. Indeed, the trivial identities
$1=2 \P \bigl \{ S_n  > \tilde{S}_n \bigr\} + \P \bigl \{ S_n  = \tilde{S}_n \bigr\}$ and
$\P \bigl \{ S_{n+1}  > \tilde{S}_n \bigr\} = \alpha \P \bigl \{ S_n  = \tilde{S}_n \bigr\} + \P \bigl \{ S_n  > \tilde{S}_n \bigr\}$ imply
\begin{equation} \label{p_n= q_n}
p_n = \frac12 + \frac{2\alpha-1}{2} \P \bigl \{ S_n  = \tilde{S}_n  \bigr\}.
\end{equation}
The asymptotic of the probability in the right-hand side is given by the classical local limit theorem.

We now turn to the proof of Theorem~\ref{MAIN 2} as it partly covers Theorem~\ref{MAIN 3}.

\begin{proof}[\bf Proof of Theorem~\ref{MAIN 2}]

In the case $1 \le r \le d-1$, the result immediately follows from Lemma~\ref{DIFF p_n^r,d 1}.
In the case of $r \ge 2d - 1$, we start with the proof of Cases (3c) and (4c). By Lemma~\ref{DIFF p_n^r,d 1}, we write
\begin{eqnarray*}
p_{n+1}^{r , d } - p_n^{r , d } &=& \alpha (1-\alpha) \biggl( \sum_{i=0}^{d-1} \Bigl( \P \bigl \{ S_r = i \bigr \} - \P \bigl \{ S_r = 2d-i-1 \bigr \} \Bigr) \bigl( q_n^{(i-d+1)} - q_n^{(i-d)} \bigr) \\
&& + \sum_{i=2d}^r \P \bigl \{ S_r = i \bigr \} \bigl( q_n^{(i-d+1)} - q_n^{(i-d)} \bigr) \biggr).
\end{eqnarray*}
The second sum is non-positive. Let us show that
$$a_i:= \frac{\P \bigl \{ S_r = 2d-i-1 \bigr \}}{\P \bigl \{ S_r = i \bigr \}} > 1, \qquad  0 \le i \le d-1$$ when $\alpha > d/(r+1)$ to prove that the first sum is negative. This statement generalizes the well known result that the (last) maximum of binomial coefficients $\P \bigl \{ S_r = i \bigr \}$ occurs at $i= [(r+1) \alpha]$.

Consider the ratio $$\frac{a_{i+1}}{a_i} = \frac{ (1-\alpha)^2 (i+1)(2d-i-1) }{\alpha^2 (r-i) (r-2d+i+2)}$$
and rewrite it in the form
$$\Bigl (\frac{1-\alpha}{\alpha} \Bigr)^2 \frac{d^2 - (d-i-1)^2}{(r-d+1)^2 - (d - i - 1)^2} =
\Bigl (\frac{1-\alpha}{\alpha} \Bigr)^2 \Bigl ( 1 - \frac{(r-d+1)^2 - d^2}{(r-d+1)^2 - (d - i - 1)^2} \Bigr)$$
to observe that this quantity increases in $i$ for $0 \le i \le d-1$ when $r \ge 2d$ and is constant when $r=2d-1$. Now as $\alpha > d/(r+1)$, in both cases we have
$$\frac{a_{i+1}}{a_i} \le \frac{a_d}{a_{d-1}} = \Bigl (\frac{1-\alpha}{\alpha} \Bigr)^2 \frac{d^2}{(r-d+1)^2}
< 1,$$
and thus $a_i$ is decreasing. Then
$$a_i \ge a_{d-1} = \frac{\P \bigl \{ S_r = d \bigr \}}{\P \bigl \{ S_r = d - 1 \bigr \}} > 1$$ because
$\alpha > d/(r+1)$.

Next we present the proof of Cases (3a), (4a), and (4b). First note that $\alpha < d/(r+1) \le 1/2$. Then by Lemma~\ref{DIFF p_n^r,d 2}, the sign of $p_{n+1}^{r,d} - p_n^{r,d}$ for large $n$ is opposite to the sign of $\tilde{P}_\alpha^{r,d}(1)=P_\alpha^{r,d}(1)=2\alpha r -2d +1$, see \eqref{P=P_tilde} and \eqref{pard(1)}.

The case $\alpha = (2d-1)/(2r)$ requires more attention. Here $P_\alpha^{r,d}(1)=0$, and the sign of $p_{n+1}^{r,d} - p_n^{r,d}$ for large $n$ coincides with that of $\frac{\partial}{\partial x} P_\alpha^{r,d}(1)$. We use the formulae $T_n'(x)= n U_{n-1}(x)$ and $T_n''(1)=(n-1)n^2(n+1)/3$, where $T_n(x)=\cos( n \arccos x)$ are Chebyshev polynomials of the first kind, to find that $U_n'(1)=n(n+1)(n+2)/3$. Hence, arguing as in the proof of \eqref{pard(1)}, we obtain
$$\frac{\partial}{\partial x} P_\alpha^{r,d}(1) = \frac13 \E (2 S_r - 2d)(2 S_r - 2d+1)(2 S_r - 2d+2).$$

As the expectation in the right-hand side equals $$\frac{\partial^3}{\partial t^3} \bigl( \E t^{2 S_r - 2d +2} \bigr) \Bigl. \Bigr|_{t=1} = \frac{\partial^3}{\partial t^3} \bigl( \bigl(1-\alpha + \alpha t^2 \bigr)^r t^{2- 2d} \bigr) \Bigl. \Bigr|_{t=1},$$ we substitute $\alpha = (2d-1)/(2r)$ and after some simplifications get
\begin{equation} \label{pard'(1)}
\frac{\partial}{\partial x} P_{(2d-1)/(2r)}^{r,d}(1) = \frac{2(2d-1)(1 + 4 d^2 +3r + 2r^2- 2d(2+3r))}{3 r^2}.
\end{equation}
In order to check that this expression is positive, note that the partial derivative in $r$ of the last factor in the numerator is equal to $4 r + 3 - 6d >0$ as $r \ge 2d$. Hence the minimal value of this factor is attained at $r=2d$ and equals $2d+1 >0$.

In the case of $d \le r \le 2d-1$, we observe the duality relation
\begin{equation} \label{duality}
p_n^{r, d}(\alpha) = 1-p_n^{r, r-d+1}(1-\alpha),
\end{equation}
which follows from $$p_n^{r, d}(\alpha) = \P \bigl\{ S_{n+r} (\alpha) \ge S_n'(\alpha) + d \bigr\} = \P \bigl\{ n+r-S_{n+r} (\alpha) \le n-S_n'(\alpha) + r - d \bigr \}$$ by comparing tails.
Here we temporarily changed the notation to stress that $p_n^{r, d}$ is a function of $\alpha$.

Now Case (3b) follows immediately while \eqref{duality} implies that Case (2) is equivalent to Case (4), which was proved above. Indeed, if $d \le  r \le 2d -2$, then $r \ge 2 d'$, where $d':= r - d +1$; and if $r \ge 2d$, then $d' \le r \le 2d' - 2$. Thus the function $(r,d) \mapsto (r, r-d+1)$ is a bijection between the sets $\{(r,d): d \le r \le 2d-2 \}$ and $\{(r,d): r \ge 2d  \}$ while $1 - d/(r+1) = d'/(r+1)$ and $1 - (2d-1)/(2r) = (2d'-1)/(2r)$.

\end{proof}

\begin{proof}[\bf Proof of Theorem~\ref{MAIN 3}]

It is well known that the maximum of $U_k(x)$ on $[0,1]$ occurs at $x=1$. Hence \eqref{P tilda}, \eqref{P=P_tilde} and \eqref{pard(1)} imply that $P_\alpha^{r, 1}(x) \le P_\alpha^{r, 1}(1) = 2\alpha r - 1$ for $x \in [0,1]$, and by Lemma~\ref{DIFF p_n^r,d 2}, the sequence $p_n^{r, 1}$ is increasing when $0 < \alpha \le 1/(2r)$. The decrease of $p_n^{r, 1}$ for $1/(r+1) < \alpha < 1$ is already covered by Case (4c) of Theorem~\ref{MAIN 2}.

For a proof of unimodality in the transitional zone it suffices to check that $P_\alpha^{r, 1}(x)$ has only one root on $[0,1]$ when $1/(2r) < \alpha \le 1/(r+1)$. Indeed, let $P_\alpha^{r, 1}(x_0)=0$ for some $x_0 \in (0,1)$, and let $p_{k+1}^{r, 1} - p_k^{r, 1} \le 0$ for some $k \ge 0$. We claim that for any $n>k$ it holds that
\begin{equation} \label{Q>Q}
Q_\alpha^n(x) \sqrt{1-x^2} P_{\alpha}^{r, d} (x) > Q_\alpha^{n-k}(x_0) Q_\alpha^k(x) \sqrt{1-x^2} P_{\alpha}^{r, d} (x), \quad x \in [0,1] \setminus \{x_0\}.
\end{equation}
First observe from \eqref{Pa} that $P_\alpha^{r, 1}(0) = -(1-2\alpha)^r <0$ while $P_\alpha^{r, 1}(1) = 2\alpha r - 1 > 0$, so $P_\alpha^{r, 1}(x) <0$ on $(0, x_0)$ and $P_\alpha^{r, 1}(x) >0$ on $(x_0, 1)$. Then we get \eqref{Q>Q} using that $Q_\alpha(x)$ is positive and increasing in $x$ on $[0,1]$. Now integrate \eqref{Q>Q} over $[0,1]$ and apply Lemma~\ref{DIFF p_n^r,d 2} to get $p_{n+1}^{r, 1} - p_n^{r, 1} < Q_\alpha^{n-k}(x_0) (p_{k+1}^{r, 1} - p_k^{r, 1}) \le 0$ for all $n > k$. Thus unimodality follows from the uniqueness of root. Note that by Lemma~\ref{DIFF p_n^r,d 1}, $$p_1^{r , 1 } - p_0^{r , 1 } = \alpha (1-\alpha) \Bigl( \P \bigl \{ S_r = 0 \bigr \} - \P \bigl \{ S_r = 1 \bigr \} \Bigr) \ge 0$$ assuring the the mode satisfies $N_{\alpha, r} \ge 1$.

Clearly, we prove that $P_\alpha^{r, 1}(x)$ has only one root on $[0,1]$ when $1/(2r) < \alpha < 1/(r+1)$ if we check that $P_\alpha^{r, 1}(x)$ is increasing on $[0,1]$ when $\alpha \le 1/(r+1)$. We claim the stronger statement: with exception of the constant term, all coefficients of the polynomial $P_\alpha^{r, 1}(x)$ are positive when $\alpha \le 1/(r+1)$.

By Theorem~\ref{MAIN TH-1}, we should show that $\frac{d^{2j}}{d t^{2j}} \bigl( t^{j-1} (1 - \alpha + \alpha t )^r \bigr) \Bigl. \Bigr |_{t=-1}$ is positive for any $1 \le j \le r-1$. The Leibnitz formula gives
$$\frac{d^{2j}}{d t^{2j}} \bigl( t^{j-1} (1 - \alpha + \alpha t )^r \bigr) \Bigl. \Bigr |_{t=-1} = \sum_{k=j+1}^{2j \wedge r} (-1)^{k-j-1} a_k,$$ where $a_k:= {2j \choose k} \frac{r ! (j-1)!}{(r - k)! (k-j-1)!} \alpha^k (1 - 2 \alpha)^{r-k},$ and it suffices to prove that $a_k$ are decreasing. We have $$\frac{a_{k+1}}{a_k} = \frac{\alpha}{1-2\alpha} \frac{(r - k) (2j - k)}{(k-j)
(k+1)},$$ which is obviously decreasing in $k$, so $$\frac{a_{k+1}}{a_k} \le
\frac{a_{j+2}}{a_{j+1}} = \frac{\alpha}{1-2\alpha} \frac{(r - j -1) (j - 1)}{j+2} <
\frac{\alpha}{1-2\alpha} (r - 2) \le \frac{r - 2}{r-1} < 1$$ when $\alpha \le
1/(r +1)$. This completes the proof of unimodality of $p_n^{r,1}$.

To prove \eqref{mode}, we use the argument similar to the proof of \eqref{sqrta} in Theorem~\ref{MAIN TH-1}. By Lemma~\ref{DIFF p_n^r,d 2} and the change $x=\cos(t/2)$, we have
$$
\frac{\pi (p_{n+1}^{r,1} - p_n^{r , 1 })}{2 \alpha (1-\alpha)} = -\int_0^\pi \varphi_\alpha^n(t) \sin^2 \frac{t}{2} \cdot P_{\alpha}^{r, d} \Bigl(\cos \frac{t}{2}\Bigr) dt,
$$
and representing $P_{\alpha}^{r, 1}(x)$ via its Taylor polynomial at $x=1$,
\begin{eqnarray*}
&& \frac{\pi (p_{n+1}^{r,1} - p_n^{r , 1 })}{2 \alpha (1-\alpha)} \\
&=& -\int_0^\pi \varphi_\alpha^n(t) \sin^2 \frac{t}{2} P_{\alpha}^{r, 1}(1)dt\\
&& +\int_0^\pi \varphi_\alpha^n(t) \sin^2 \frac{t}{2} \biggl [\frac{\partial}{\partial x} P_{\alpha}^{r, 1}(1) \Bigl(1 - \cos \frac{t}{2}\Bigr)- \frac12 \frac{\partial^2}{\partial x^2} P_{\alpha}^{r, 1}(\theta(t)) \Bigl(1 - \cos \frac{t}{2}\Bigr)^2 \biggr]dt\\
&=:& -I_1(\alpha, n) + I_2(\alpha, n),
\end{eqnarray*}
where $\cos (t/2) \le \theta(t) \le 1$.

Now assume $n=n(\alpha) \to \infty$ as $\alpha \downarrow 1/(2r)$. We claim that
\begin{equation} \label{I}
I_1(\alpha, n(\alpha)) \sim \frac{\sqrt{\pi} r^3}{2 (2r-1)^{3/2}} \cdot \frac{2\alpha r - 1}{n(\alpha)^{3/2}}, \quad I_2(\alpha, n(\alpha)) \sim \frac{\sqrt{\pi} (r-1) r^3}{4 (2r-1)^{3/2}} \cdot \frac{1}{n(\alpha)^{5/2}} \quad \mbox{as} \quad \alpha \downarrow \frac{1}{2r}.
\end{equation}
Indeed, arguing as in \eqref{dif sim} and \eqref{Lebesgues} and using \eqref{phi lim}, which of course holds when $\alpha \to 1/(2r)$, we get
$$I_1(\alpha, n(\alpha)) \sim \frac{P_{\alpha}^{r, 1}(1)}{4n^{3/2}} \int_0^{\delta n^{1/2}} \varphi_\alpha^n \Bigl(\frac{s}{\sqrt{n}}\Bigr) s^2 ds \sim \frac{P_\alpha^{r, 1}(1)}{4n^{3/2}} \int_0^\infty e^{-\frac{2r-1}{4r^2} s^2} s^2 ds$$
by the dominated convergence theorem. It only remains to compute the integral and use the equation $P_{\alpha}^{r, 1}(1) = 2 \alpha r -1$ to get the first part of \eqref{I}. Similarly, $$I_2(\alpha, n(\alpha)) \sim
\frac{\frac{\partial}{\partial x} P_{\alpha}^{r, 1}(1)}{32 n^{5/2}} \int_0^{\delta n^{1/2}} \varphi_\alpha^n \Bigl(\frac{s}{\sqrt{n}}\Bigr) s^4 ds \sim \frac{\frac{\partial}{\partial x} P_{1/(2r)}^{r, 1}(1)}{32n^{5/2}} \int_0^\infty e^{-\frac{2r-1}{4r^2} s^2} s^4 ds$$ as the term coming from the second derivative of $P_{\alpha}^{r, 1}(x)$ gives a lower order contribution, and we get the second relation in \eqref{I} by computing the integral and using \eqref{pard'(1)}.

Clearly, $N_{\alpha, r}  \to \infty$ as $\alpha \downarrow 1/(2r)$, and \eqref{I} implies $I_k(\alpha, N_{\alpha, r}  -1) \sim
I_k(\alpha, N_{\alpha, r} )$, $k=1,2$. By the definition of $N_{\alpha, r} $, we have $I_2(\alpha, N_{\alpha, r}  -1) \ge I_1(\alpha, N_{\alpha, r}  -1)$ and $I_1(\alpha, N_{\alpha, r} ) > I_2(\alpha, N_{\alpha, r} )$, hence $$I_1(\alpha, N_{\alpha, r} ) \sim I_2(\alpha, N_{\alpha, r} ) \quad \mbox{as} \quad \alpha \downarrow \frac{1}{2r}.$$ We now apply \eqref{I} again to get \eqref{mode}.

For \eqref{maximum}, we argue in the same manner as above. Omitting the details, we get
\begin{eqnarray*}
2 \pi \Bigl(p_{N_{\alpha, r}} - \frac12 \Bigr) &\sim& \frac{P_\alpha^{r, 1}(1)}{N_{\alpha, r}^{1/2}} \int_0^\infty e^{-\frac{2r-1}{4r^2} s^2} ds -\frac{\frac{\partial}{\partial x} P_{1/(2r)}^{r, 1}(1)}{8N_{\alpha, r}^{3/2}} \int_0^\infty e^{-\frac{2r-1}{4r^2} s^2} s^2 ds,
\end{eqnarray*}
and the required relation follows.
\end{proof}

\section{Additional remarks and related questions}

Here we mention some issues that
deserve more attention than we can provide. Then we explain our initial interest to the problem.

Simulations show that as in the case $p_n^{r,1}$ with $r \ge 2$, the sequence $p_n^{r,d}$ with $r \ge 2d$ is increasing when $0 < \alpha < (2d-1)/(2r)$ and unimodal in the transitional phase $(2d-1)/(2r) < \alpha \le d/(r+1)$. The proofs presented here do not admit a reasonable generalization even for $d=2$. For example, simulations show that the coefficient of $P_\alpha^{r, 2}(x)$ at $x$ is negative and all the other coefficients are positive. This could be neither easily proved nor used for a proof similar to the one given here for $d=1$ as $P_\alpha^{r, 2}(x)$ has two roots on $[0,1]$ for some $\alpha$. Once unimodality in the transitional phase is established for general $d$, the asymptotic of the mode $N_\alpha^{r, d}$ could be found by exactly the same argument as used for \eqref{mode}.




Another open question is to verify that the sequence $p_n^{r,d}$ is convex/concave in the monotone modes, as observed by simulations. By the same argument as in the proof of Theorem~\ref{MAIN TH-1}, we easily get convexity (concavity, to be precise) of $p_n^{r,1}$ when $0 < \alpha \le 1/(2r)$. Unfortunately, the same method does not seem to work for $1/(r+1) < \alpha <1$.

Next we should mention that, after a preliminary draft of this paper was finished, we found that a related problem was studied in \cite{AWW}.
Their analysis uses the multivariate form of Zeilberger's algorithm to find representations and estimates
based on a linear combination of two Legendre polynomials. Also, after submitting this paper, Tamas Lengyel and the referee pointed out the paper \cite{Le11} on approximating point spread distributions. Our Fourier analytic method can also be used to simplify some of the arguments in both \cite{AWW} and \cite{Le11},
and provide integral representations for general $r$, $d$, $p$ and $q$, see \cite{DLS} for details.

Finally, let us explain the origins of our interest to the problem considered in this paper. We say that a random variable $Z$ has a double-sided exponential distribution if it has a density $\rho(x)$ of the form $\rho(x)= \alpha a e^{-a x}$ for $x>0$ and $\rho(x)= (1- \alpha) b e^{b x}$ for $x<0$, where $a, b>0$ and $\alpha \in (0,1)$. Such a specific random variable has very nice properties from the point of view of fluctuation theory, see \cite[Ch. VI.8]{Feller}. $\E Z=0$ means $b=\frac{1-\alpha}{\alpha}a$ so any centered double-sided exponential variable could be written in the form $Z = \frac1a \bigl(X Y - \frac{\alpha}{1-\alpha}(1-X)Y \bigr)$, where $X$ is a Bernoulli random variable that equals $1$ and $0$ with probability $\alpha$ and $1-\alpha$, respectively, $Y$ is a standard exponential random variable with density $e^{-x}$ for $x>0$, and $X$ and $Y$ are independent.

Consider independent identically distributed double-sided exponential random variables $Z_n$. Initially,
our goal was to show that the sequence $\P \bigl \{ Z_1 + \dots + Z_n > 0 \bigr \} - 1/2$ does not change its sign as $n$ increases. This property was verified by simulations, and a rigorous proof was required. One of the possible approaches is to show that the sequence $\P \bigl \{ Z_1 + \dots + Z_n > 0 \bigr \}$ is monotone since by the central limit theorem, its limit equals $1/2$.

Introducing independent sequences of standard exponential random variables $Y_n$ and Bernoulli random variables $X_n$, we  write
\begin{eqnarray*}
\P \bigl \{ Z_1 + \dots + Z_n > 0 \bigr \} &=& \sum_{k=1}^{n} {n \choose k} \alpha^k (1-\alpha)^{n-k} \P \Bigl \{ Y_1 + \dots + Y_k > \frac{\alpha}{1-\alpha} \Bigl ( Y_{k+1} + \dots Y_n \Bigr ) \Bigr \}\\ &=& \sum_{k=1}^{n}  {n \choose k} \alpha^k (1-\alpha)^{n-k} \P \Bigl \{ Y_1 + \dots + Y_k > \alpha \bigl ( Y_1 + \dots Y_n \bigr ) \Bigr \}\\
&=& \sum_{k=1}^{n-1}  {n \choose k} \alpha^k (1-\alpha)^{n-k} \P \bigl \{ U_{k, n-1}> \alpha \bigr \} + \alpha^n,\\
\end{eqnarray*}
where we used the well known (Karlin~\cite{K}) representation of uniform order statistics in terms of exponential random variables.

Recall the appropriate definition. Let $U_1, \dots, U_m$ be independent random variables that are uniformly distributed on $[0,1]$, and arrange them in the ascending order. The $k$-th element of this new sequence is called the $k$-th order statistics of $U_1, \dots, U_m$ and denoted by $U_{k, m}$. By definition, $U_{1, m} \le \dots \le U_{m, m}$ but all the inequalities are strict with probability one.

Note that $\P \bigl \{ U_{k, m}> \alpha \bigr \} = \P \bigl \{ S_m < k \bigr \}$, hence
$$\P \bigl \{ Z_1 + \dots + Z_n > 0 \bigr \} = \sum_{k=1}^{n-1} {n \choose k} \alpha^k (1-\alpha)^{n-k} \P \bigl \{ S_{n-1} <k \bigr \} + \alpha^n = \P \bigl \{ S_n' > S_{n-1} \bigr \} = p_{n-1}.$$ We see that monotonicity of $\P \bigl \{ Z_1 + \dots + Z_n > 0 \bigr \}$ is exactly the monotonicity of $p_n$ given by Theorem~\ref{MAIN TH-2}.

{\bf Acknowledgement:} The authors would like to thank Tamas Lengyel and the referee for pointing out the paper \cite{Le11}
and their insightful remarks and suggestions.

\end{document}